\documentclass[12pt]{article}
\usepackage{amsmath,amssymb,amsthm,amscd}
\usepackage[all]{xy}
\usepackage{graphicx}

\newcommand{\N}{\mathbb{N}}
\newcommand{\Z}{\mathbb{Z}}
\newcommand{\Q}{\mathbb{Q}}
\newcommand{\R}{\mathbb{R}}

\newcommand{\T}{\mathbb{T}}

\newtheorem{defn}{Definition}[section]
\newtheorem{thm}[defn]{Theorem}
\newtheorem{ex}[defn]{Example}
\newtheorem{prop}[defn]{Proposition}
\newtheorem{cor}[defn]{Corollary}
\newtheorem{lemma}[defn]{Lemma}

\newtheorem{rem}[defn]{Remark}

\title{$\Z^{d}$-odometers and cohomology}

\author{Thierry Giordano\thanks{Supported in part by a
grant from NSERC, Canada}, \\
Department of Mathematics and Statistics, \\
University of Ottawa,\\
585 King Edward, Ottawa, Ontario, Canada K1N 6N5
 \and
 Ian F. Putnam\thanks{Supported in part by a
grant from NSERC, Canada},\\
Department of Mathematics and Statistics,\\
University of Victoria,\\
Victoria, B.C., Canada V8W 3R4
\and 
Christian F. Skau\thanks{Supported in
part by the Norwegian Research Council}, \\
Department of Mathematical Sciences, \\
Norwegian
University of Science and Technology (NTNU), \\
N-7034 Trondheim, Norway}

\begin{document}
\maketitle

\begin{abstract}
Cohomology for actions of  free abelian groups on the Cantor set has
(when endowed with an order structure) provided a complete 
invariance for orbit equivalence. In this paper, we study 
a particular class of actions of such groups called odometers 
(or profinite actions) and
investigate their cohomology. We show that for a free, minimal 
$\Z^{d}$-odometer, the first cohomology group provides a complete 
invariant for the action up to conjugacy. This is in contrast with
 the situation for orbit 
equivalence where it is the cohomology in dimension $d$ which provides
the invariant. We also consider classification up to isomorphism 
and continuous orbit equivalence.
\end{abstract}

\section{Introduction}

We recall some basic definitions from dynamical systems (see 
\cite{Gla}, \cite{Ku} or \cite{Wa}).
  
 We say that $(X, \varphi)$ is 
 an action of a group $G$ if $X$ is a topological space and, for every 
 $g$ in $G$, $\varphi^{g}: X \rightarrow X$ is a homeomorphism 
 satisfying the condition that,
 for all $g,h$ in $G$, we have
  $\varphi^{g} \circ \varphi^{h} = \varphi^{gh}$. 
  We will only consider groups $G$ which are countable.
 
An action $(X, \varphi)$ of $G$ is \emph{free},
 if whenever $g$ is in $G$ and
$x$ in $X$ satisfy $\varphi^{g}(x) = x$, then $g=e$.
 We say that the action
is \emph{minimal} if the only closed subsets $Y$ of $X$ such that 
$\varphi^{g}(Y) = Y$, for all $g$ in $G$, are the empty
 set and $X$.

  If 
   $(X_{1},  \varphi_{1})$ and 
    $(X_{2}, \varphi_{2})$ are  actions of the groups $G_{1}$
    and $G_{2}$, respectively, then  
 $(X_{1} \times X_{2}, \varphi_{1} \times \varphi_{2} )$ is an 
 action of $G_{1} \times G_{2}$.

We also recall that if $(X, \varphi)$ and $(Y, \psi)$ are
 actions of the group
$G$, a factor map $\pi: (X, \varphi) \rightarrow (Y, \psi)$ is 
a continuous surjection, $\pi: X \rightarrow Y$ such that
 $\pi \circ \varphi^{g} = \psi^{g} \circ \pi$, for every
  $g$ in $G$.

An interesting class of systems where the space 
$X$ is compact and totally disconnected arises if  the group is assumed to
be \emph{residually finite}: there is a decreasing sequence 
of subgroups $G \supseteq G_{1} \supseteq G_{2} \supseteq \cdots $, each having finite index , $[G: G_{n}] < \infty$ and whose intersection
is just the identity. Then the group $G$ acts in an obvious way
on each finite quotient 
 space $G/G_{n}$ and on the inverse limit 
 of the system
 \[
 G/G_{1} \leftarrow G/G_{2} \leftarrow \cdots 
 \]
  Such systems have a long and rather complex history, particularly
  in ergodic theory. We refer the reader to \cite{GNS} and 
  \cite{Nek} for a full discussion. In the topological category, we 
  refer the reader to \cite{CP} and \cite{Do}. Such systems often 
  appear under the category of profinite  completions, but we shall refer
  to them as $G$-odometers, simply because the terminology of 'odometer' is 
  standard for the $G=\Z$ case.

Our main interest here will be in  $\Z^{d}$-odometers, $ d \geq 1$.
We use $e_{1}, \ldots, e_{d}$ for the standard set of generators
 of $\Z^{d}$. We let $\varepsilon_{1}, \ldots, \varepsilon_{d}$ denote
  the dual basis; that is, these are group homomorphisms
   from $\Z^{d}$ to $\Z$. We use $< \cdot, \cdot>$ to 
   denote the usual inner product on $\R^{d}$.

\begin{defn}
\label{intro:5}
Let $(X, \varphi)$ be an action of the group 
$G$ and $(Y, \psi)$ be an action of the group $H$.
\begin{enumerate}
\item 
If $G = H$, 
a \emph{conjugacy}, $h : (X, \varphi) \rightarrow (Y, \psi)$, is 
a homeomorphism $h: X \rightarrow Y$ such that 
$h \circ \varphi^{g} = \psi^{g} \circ h$, for all $g$ 
in $G$. If such a map exists, we say that 
 $(X, \varphi)$ and $(Y, \psi)$ are \emph{conjugate}.
 \item 
 An \emph{isomorphism} between the actions is 
 a pair
$(h, \alpha)$, where $h: X \rightarrow Y$ is
a homeomorphism and $\alpha:G \rightarrow H$  is a group 
isomorphism such that 
$h \circ \varphi^{g} = \psi^{\alpha(g)} \circ h$, for all $g$ in $G$.
 If such a pair exists, we say that 
 $(X, \varphi)$ and $(Y, \psi)$ are \emph{isomorphic}.
 \end{enumerate}
\end{defn}

Of course, in the second definition, even when $G=H$, we allow
$\alpha$ to be non-trivial.

We will also be considering orbit equivalence between our systems
\cite{GPS1}.
The terminology is a reflection of the fact that in a
$G$-action, $(X, \varphi)$, for any point $x$ in $X$, the 
set $\{ \varphi^{g}(x) \mid g \in G \}$ is called the orbit of $x$ under 
$\varphi$.

\begin{defn}
\label{intro:10}
Let $(X, \varphi)$ be an action of the group $G$  and 
$(Y, \psi)$ be an action of the group $H$. We say they are 
\emph{orbit  equivalent} if there is a homeomorphism 
$h : X \rightarrow Y$ such   
\[
h \{ \varphi^{g}(x) \mid g \in G \} 
= \{ \psi^{g'}(h(x)) \mid g' \in H \},
\]
for all $x$ in $X$. The function $h$ is called an 
orbit equivalence.
\end{defn}

Note that in the definitions of conjugacy and 
isomorphism, the groups must be isomorphic, which 
is not the case here.

 Suppose that $(X, \varphi)$ and $(Y, \psi)$ are orbit
  equivalent and $h$ is a 
map as in the definition. If we assume that both actions are free,
 then there are unique functions
 $\alpha: X \times G \rightarrow H$
  and $\beta: Y \times H \rightarrow G$ such that 
  \[
  h(\varphi^{g}(x)) = \psi^{\alpha(x, g)}(h(x)),
  \]
  for all $x$ in $X$, $g$ in $G$ and 
  \[
  h^{-1}(\psi^{g'}(y)) = \varphi^{\beta(y,g')}(h^{-1}(y)),
  \]
  for all $y$ in $Y$ and $g' $ in $H$.
 These functions are usually called the \emph{orbit cocycles}
 associated with $h$. It is important to note that 
 despite the continuity properties of $\varphi, \psi$ and $h$, 
 these functions need not be continuous. 
 It is then a  natural notion to require some type of continuity.
 
 \begin{defn}\cite{Li}
 \label{intro: 15}
 If $(X, \varphi)$ is a free action of $G$ 
  and $(Y,  \psi)$ is a free action of $H$, we say they are 
 \emph{continuously orbit equivalent} if there is 
 an orbit equivalence, $h$, between them, whose cocycles, 
 $\alpha$ and $\beta$, are both continuous (with the usual
 topologies on $X, Y$, the discrete topologies on
  $G, H$ and the product topologies on 
   $X \times G$ and $Y \times H$).
 \end{defn}

It is also worth 
noting that for minimal actions of $\Z$ on the Cantor set,
 the somewhat peculiar 
 property that $\alpha( \cdot, 1)$ and $\beta( \cdot, 1)$ each 
 have at most one point of discontinuity is called \emph{strong} orbit 
 equivalence. (For more information, see \cite{GPS1}.)

 It is probably worth noting for the record the fairly obvious
 facts: 
 conjugacy implies isomorphism, which implies
 continuous orbit equivalence, which implies orbit equivalence.

 Our main results will be a classification of $\Z^{d}$-odometers
 up to conjugacy, isomorphism, orbit equivalence and 
 continuous orbit equivalence. In fact, some  results 
 along these lines have already 
 been obtained in \cite{CP}, \cite{CM} and \cite{Li}. 
 What is new in this paper is three features.
 The first is an alternate description
 of the construction. Instead of starting with a sequence
 \[
 \mathcal{G}: \Z^{d} = G_{1} \supseteq G_{2} \supseteq \cdots
 \]
 we start with a single group $\Z^{d} \subseteq H \subseteq \Q^{d}$.
 The passage between the two is obtained by duality, either 
 Pontryagin or in the sense of dual lattice.
More specifically, $H$ is obtained from 
$\mathcal{G}$ as $H = \cup_{n=1}^{\infty}  G_{n}^{*}$, where 
$G^{*}$ denotes the dual lattice of $G$, but we explore this
in more detail in the next section. Our associated $\Z^{d}$-odometer
is denoted by $(Y_{H}, \psi_{H})$.
There is a small benefit here in that we replace the data of a sequence, 
$\mathcal{G}$,
by a single object, namely $H$.

The second novelty is the computation of the first
cohomology  group, $H^{1}(Y_{H}, \psi_{H})$. Specifically, we show 
in Theorem \ref{odco:7} that this 
is isomorphic to $H$. Even better, 
using the fact that the odometer has a unique invariant 
probability measure $\mu$, we describe a natural map
$\tau^{1}_{\mu}: H^{1}( Y_{H}, \psi_{H}) \rightarrow \R^{d}$.
(It is worth noting that the map is coordinate dependent, 
however.) In Theorem \ref{odco:8}, we
 show that, for $d=1, 2$, the range is exactly $H$ and the 
map is an isomorphism.

The third novelty is that, building on the first two, we 
are able to give a quite simple classification of 
$\Z^{d}$-odometers up to conjugacy, 
isomorphism, continuous orbit equivalence
and orbit equivalence in terms of the group $H$.
 We summarize the results of  Corollaries
\ref{class:10} and \ref{class:16} and Theorem 
\ref{class:25} in the following theorem, Theorem \ref{intro_25}.
Before stating this, we need a bit of notation.
 It must be clear that 
that in cases of interest, $\Z^{d}$ has infinite index in $H$. 
Never-the-less, we find it useful to describe the 
index data in the following form.

\begin{defn}
\label{intro:27}
Let $ \Z^{d} \subseteq H \subseteq \Q^{d}$. We define
 the \emph{superindex} of $\Z^{d}$ in $H$ by 
\[
[[H: \Z^{d}]] = \{  [H': \Z^{d}] \mid 
\Z^{d} \subseteq H' \subseteq H, [H': \Z^{d}] < \infty \}.
\]
\end{defn}

Recall that  $GL_{d}(\Q)$ consists of the invertible $d \times d$ matrices
with rational entries, while
$GL_{d}(\Z)$ consists of the $d \times d$ matrices with integer entries and
determinant $\pm 1$. Also recall that $SL_{d}(\Q)$ and $SL_{d}(\Z)$
consist of the respective subgroups with matrices of determinant one.

\begin{thm}
\label{intro_25}
Let $\Z^{d} \subseteq H \subseteq \Q^{d}$  and 
 $\Z^{d'} \subseteq H' \subseteq \Q^{d'}$ be two groups and assume 
each is dense in $\Q^{d}$ and $\Q^{d'}$, respectively.
\begin{enumerate}
\item If $d = 1, 2$, then the $\Z^{d}$-actions 
$(Y_{H}, \psi_{H})$ and 
$(Y_{H'}, \psi_{H'})$ are conjugate if and only if 
$H=H'$.
\item If $d, d' = 1, 2$, then the $\Z^{d}$-action 
$(Y_{H}, \psi_{H})$ and the $\Z^{d'}$-action
$(Y_{H'}, \psi_{H'})$ are isomorphic  if and only if  $d=d'$ and
there is $\alpha$ in $GL_{d}( \Z)$ such that $\alpha H=H'$.
\item 
If $d, d' = 1, 2$, then the $\Z^{d}$-action 
$(Y_{H}, \psi_{H})$ and the $\Z^{d'}$-action
$(Y_{H'}, \psi_{H'})$ are continuously orbit equivalent
  if and only if  $d=d'$ and
there is $\alpha$ in $GL_{d}( \Q)$ with 
$\det(\alpha) = \pm 1$ such that $\alpha H=H'$.
\item 
The $\Z^{d}$-action 
$(Y_{H}, \psi_{H})$ and the $\Z^{d'}$-action
$(Y_{H'}, \psi_{H'})$ are  orbit equivalent
  if and only if $[[H:\Z^{d}]] = [[H':\Z^{d'}]]$.
\end{enumerate}
\end{thm}

We note in Corollary \ref{class:30} that
 in the case $d=d'=1$, the four conditions
are  all equivalent. In Example
\ref{class:35}, we also give examples to show that, aside from the 
obvious implications we noted earlier, the four conditions are distinct, 
although all of these were provided initially 
by others.

\section{$\Z^{d}$-odometers}

Let $d$ be a positive integer. We say that 
$\mathcal{G} = \{ G_{1}, G_{2}, \ldots \}$ is a 
decreasing sequence of finite-index subgroups of $\Z^{d}$ if
we have \begin{enumerate}
\item $\Z^{d} = G_{1} \supseteq G_{2} \supseteq \cdots$,
\item 
$[ \Z^{d}: G_{n}] < \infty$, for all $n \geq 1$.
\end{enumerate}

If $G$ is any  subgroup of $\Z^{d}$, 
we let 
$\varphi_{G}$ denote the $\Z^{d}$-action on 
$\Z^{d}/G$ given by
\[
\varphi_{G}(k)( l + G) = k+l+ G,  \text{ for all } k,l \in \Z^{d}.
\]
Given the decreasing sequence above,
 the obvious quotient map \newline 
 $q_{n}: \Z^{d}/G_{n+1} \rightarrow \Z^{d}/G_{n}$ is then a factor map and we define
$(X_{\mathcal{G}}, \varphi_{\mathcal{G}})$ to be 
the inverse limit of the systems
\[
(\Z^{d}/G_{1},\varphi_{G_{1}}) \stackrel{q_{1}}{\longleftarrow} 
(\Z^{d}/G_{2},\varphi_{G_{2}}) \stackrel{q_{2}}{\longleftarrow} 
\cdots 
\]
The natural map from 
$(X_{\mathcal{G}}, \varphi_{\mathcal{G}})$ 
to $(\Z^{d}/G_{n},\varphi_{G_{n}})$
is denoted $\pi_{n}$, for $n \geq 1$.

\begin{defn}
\label{odom:5}\cite{Co}
A $\Z^{d}$-odometer is any system
 $(X_{\mathcal{G}}, \varphi_{\mathcal{G}})$, where $\mathcal{G}$  is 
 a decreasing sequence of finite-index 
 subgroups of $\Z^{d}$.
\end{defn}

The proof of the following is direct and we leave it to the reader.

\begin{thm}
\label{odom:8}
Let $\mathcal{G}$ be  a decreasing
 sequence of finite-index subgroups of $\Z^{d}$.
\begin{enumerate}
\item 
If $G_{n} \neq G_{n+1}$ for infinitely many $n$, then $ X_{\mathcal{G}}$ 
is a Cantor set.
\item 
$(X_{\mathcal{G}}, \varphi_{\mathcal{G}})$
is minimal.
\item The action is free if and only if 
$\cap_{n=1}^{\infty} G_{n } = \{ 0 \}$. 
\item
There is a unique 
$\varphi_{\mathcal{G}}$-invariant 
probability measure $\mu_{\mathcal{G}}$ on 
$X_{\mathcal{G}}$ which satisfies
\[
\mu_{\mathcal{G}}(\pi_{n}^{-1}\{ k + G_{n}\})
 = [\Z^{d}: G_{n}]^{-1},
 \]
 for all $n \geq 1$, $k$ in $\Z^{d}$.
 \item The formula
 \[
 d_{\mathcal{G}}(x, y) = \sup\{0, n^{-1} \mid n \geq 1,
 \pi_{n}(x) \neq \pi_{n}(y) \}
 \]
 for $x, y$ in $X_{\mathcal{G}}$ defines a metric in which 
 $\varphi_{\mathcal{G}}$ is isometric.
 \end{enumerate}
\end{thm}

Our aim here is to present $\Z^{d}$-odometers in a 
slightly different fashion, although the difference
is rather cosmetic. A crucial feature is that the group 
$\Z^{d}$ is abelian, as we will use Pontryagin duality 
in an essential way.

 We begin with 
a group 
 $\Z^{d} \subseteq H \subseteq \Q^{d}$.
 It follows that 
 \[
 H/ \Z^{d} \subseteq \Q^{d} / \Z^{d} \subseteq \R^{d} / \Z^{d} \cong \T^{d}.
 \]
 We let $\rho$ denote the inclusion map 
 of $H/Z^{d}$ in 
 $\T^{d}$; that is,
 \[
 \rho ( (r_{1}, \ldots, r_{d}) + \Z^{d}) = 
( e^{2 \pi  i r_{1}}, \ldots, e^{2 \pi i r_{d}}), r \in H.
\]
For any locally compact abelian group $K$, we let 
$\widehat{K}$ denote its Pontryagin dual \cite{HR}. We let
 $Y_{H} = \widehat{ H / \Z^{d}}$. The groups $H$ and $H / \Z^{d}$
  are given
 the discrete topology so that $Y_{H}$ is compact. Since 
 $H \subseteq \Q^{d}$, the quotient $H / \Z^{d}$ 
 is torsion, so $Y_{H}$ is totally disconnected.

We suppress the natural isomorphism 
$\widehat{\T^{d}} \cong \Z^{d}$ and consider
\[
\widehat{\rho}: \Z^{d} \rightarrow \widehat{ H / \Z^{d}}.
\]
We then obtain an action of $\Z^{d}$ on $Y_{H}$, which 
we denote by $\psi_{H}$, by 
\[
\psi_{H}^{n}(x) = x + \widehat{\rho}(n), n \in \Z^{d}, x \in Y_{H}.
\]
More specifically, if $x: H/\Z^{d} \rightarrow \T$ is a 
group homomorphism, then 
\[
\psi_{H}^{n}(x)(h + \Z^{d}) = x(h) e^{2 \pi i <h,n>},
h \in H, n \in \Z^{d}.
\]

If we rewrite this, using $<\!< \cdot, \cdot > \!>$ to denote 
the pairing between $H/\Z^{d}$ and its dual, then we have 
\[
<\!< h + \Z^{d}, \psi^{n}_{H}(x) > \!> = 
e^{2 \pi <h, n>} <\!< h + \Z^{d}, x > \!>, 
\]
for every $h$ in $H$, $n $ in $\Z^{d}$ and $x$ in $Y_{H}$.
In other words, the function \newline 
$<\!< h + \Z^{d}, \cdot > \!>$ 
is a continuous eigenfunction for the action with eigenvalue
$(e^{2 \pi h_{1}}, \ldots,  e^{2 \pi h_{d}})$. 
In particular, the spectrum \cite{Wa} of the 
action is the set $\{ (e^{2 \pi h_{1}}, \ldots,  e^{2 \pi h_{d}}) 
\mid h \in H \}$.

 We make a few simple observations on this construction.
 
 \begin{prop}
 \label{odom:11}
 \begin{enumerate}
 \item  If $\Z^{d} \subseteq H \subseteq \Q^{d}$, then 
 the action $(Y_{H}, \psi_{H})$ is free if and only 
 if $H$ is dense in $\Q^{d}$.
  \item If $\Z^{d} \subseteq H \subseteq \Q^{d}$, then 
  $\# Y_{H} = [H: \Z^{d}]$.
  \item If $\Z^{d} \subseteq H \subseteq H' \subseteq \Q^{d}$, then 
 there is a natural factor map from $(Y_{H'}, \psi_{H'})$ to 
  $(Y_{H}, \psi_{H})$.
  \item If  $\Z^{d} \subseteq H_{1} \subseteq H_{2}  \subseteq \cdots \Q^{d}$
 then the inverse limit of 
 \[
  \lim (Y_{H_{1}}, \psi_{H_{1}})  
\stackrel{\widehat{i_{1}}}{\longleftarrow}
 (Y_{H_{2}}, \psi_{H_{2}}) 
 \stackrel{\widehat{i_{2}}}{\longleftarrow} \cdots
 \]
 is conjugate to $(Y_{H}, \psi_{H})$, where 
 $H = \cup_{n} H_{n}$.
 \end{enumerate}
 \end{prop}

 To  analyze the system, 
 $(Y_{H}, \psi_{H})$, 
  we say that a sequence of subgroups,
$ H_{1}, H_{2}, \ldots $ 
is an increasing sequence of finite-index 
extensions of $\Z^{d}$ 
if \begin{enumerate}
\item $\Z^{d} = H_{1} \subseteq H_{2} \subseteq \cdots$,
\item 
$[ H_{n}: \Z^{d}] < \infty$, for all $n \geq 1$.
\end{enumerate}
Observe that, for any group $\Z^{d} \subseteq H \subseteq \Q^{d}$, 
there exists such a sequence with union $H$ by simply taking 
$H_{n} = (\frac{1}{n!} \Z )^{d} \cap H, n \geq 1$. From this point, the 
last proposition gives us a rather complete description of the 
action $(Y_{H}, \psi_{H})$.

The link between the two constructions
 which we have described,
$  \Z^{d} = G_{1} \supseteq G_{2} \supseteq \cdots$ and 
$ \Z^{d} = H_{1}  \subseteq H_{2} \subseteq \cdots$ is given 
by duality. Recall that a subgroup $K \subseteq \R^{d}$ is a 
\emph{lattice} if it is discrete and $\R^{d}/K$ is compact. 
In particular, this
holds if $\Z^{d} $ is a finite index subgroup
of $K$. The dual lattice $K^{*}$ is defined
by 
\[
K^{*} = \{ g \in \R^{d} \mid < k, g> \in \Z, \text{ for all }
k \in K \}.
\]
It is a simple matter to check that if $H \subseteq \Q^{d}$
 is a lattice, then
$H^{*} \subseteq \Q^{d}$ also.

\begin{lemma}
\label{odom:14}
Let $\Z^{d} \subseteq K \subseteq \Q^{d}$ be a subgroup
 with $[K : \Z^{d} ] < \infty$. Then 
 there is a conjugacy 
 \[
 h_{K}: (Y_{K}, \psi_{K}) \rightarrow (\Z^{d}/K^{*}, \varphi_{K^{*}}).
 \]

 Moreover, if $\Z^{d} \subseteq K_{1} \subseteq K_{2} \subseteq \Q^{d}$ 
 are subgroups with 
 $[K_{2} : \Z^{d} ] < \infty$,  let $i$ denote 
 the inclusion of $K_{1}/ \Z^{d}$ in $K_{2}/ \Z^{d}$. 
 It is clear that $K_{1}^{*} \supseteq K_{2}^{*}$ and we let 
 $q$ denote the natural quotient map from
 $\Z^{d}/ K_{2}^{*}$ to  $\Z^{d}/ K_{1}^{*}$ . The following
  diagram commutes:
 
 \hspace{2cm}
 \xymatrix{
 (Y_{K_{2}}, \psi_{K_{2}}) \ar[r]^{\widehat{i}}
   \ar[d]_{h_{K_{2}}} &
  (Y_{K_{1}}, \psi_{K_{1}})    \ar[d]_{h_{K_{1}}} \\
  (\Z^{d}/K_{2}^{*}, \varphi_{K_{2}^{*}}) 
  \ar[r]^{q}  &
    (\Z^{d}/K_{1}^{*}, \varphi_{K_{1}^{*}}) } 
\end{lemma}

The obvious immediate consequence is the following.

\begin{thm}
\label{odom:17}
Let $H$ be a group with
$\Z^{d} \subseteq H \subset \Q^{d}$.
 If $H_{1}, H_{2}, \ldots$ is any increasing  sequence 
of finite-index extensions of 
$\Z^{d}$ with union $H$, then 
$\mathcal{G} = \{ H_{1}^{*}, H_{2}^{*}, \ldots \}$ is a 
decreasing sequence of finite-index subgroups of
$\Z^{d}$ and the  $\Z^{d}$-systems
$(Y_{H}, \psi_{H})$ and 
$(X_{\mathcal{G}}, \varphi_{\mathcal{G}})$ are conjugate.
\end{thm}

We also note the following, which establishes that the 
correspondence is actually bijective.

\begin{thm}
\label{odom:20}
Let $\mathcal{G}   = \{ G_{1}, G_{2}, \ldots \}$ be
 a 
decreasing sequence of finite-index subgroups of
$\Z^{d}$. Then the group
 \[
 H = \cup_{n = 1} ^{\infty} G_{n}^{*}
 \]
 satisfies  $\Z^{d} \subseteq H \subseteq \Q^{d}$.
  Moreover, 
  the  $\Z^{d}$-systems
$(Y_{H}, \psi_{H})$ and 
$(X_{\mathcal{G}}, \varphi_{\mathcal{G}})$ are conjugate. 
\end{thm}

In short, we have an equivalent formulation
for $\Z^{d}$-odometers \newline 
parametrized by the group
$H$ instead of the sequence $\mathcal{G}$. 
 We believe that the 
parametrization by the group 
$H$ is more natural. At this point, it is slightly simpler,
being given as a single group, rather than a sequence
of groups. That is rather trivial; we hope to 
make the case more convincing in subsequent sections.

Next, we list two relatively simple results which are worth 
noting.

\begin{prop}
\label{odom:24}
Let $d_{1}, d_{2} \geq 1$ and
 $\Z^{d_{i}} \subseteq H^{i} \subseteq \Q^{d_{i}}$ 
 be groups for $ i = 1,2$. Then 
$(Y_{H^{1} \oplus H^{2}}, \psi_{H^{1} \oplus H^{2}})$ 
is conjugate to the product 
system \newline
$(Y_{H^{1}} \times Y_{ H^{2}}, \psi_{H^{1} }
 \times \psi_{H^{2}})$ as $\Z^{d_{1} + d_{2}}$-systems.
\end{prop}

One other nice feature of our parameterization of odometers
by $\Z^{d} \subseteq H \subseteq \Q^{d}$ is that it makes
the difference between conjugacy and isomorphism 
relatively easy to describe, as follows.

\begin{prop}
\label{odom:25}
Let $\Z^{d} \subseteq H \subseteq \Q^{d}$ be a group and 
let $\alpha$ be in $GL_{d}(\Z)$.
Then  
the $\Z^{d}$-odometers $(Y_{H}, \psi_{H})$ and 
$(Y_{\alpha H}, \psi_{\alpha H})$ are isomorphic
via the automorphism of $\Z^{d}$ sending $n \in \Z^{d}$ to 
$ \alpha n$.

Conversely, if $ \Z^{d} \subseteq H, H' \subseteq \Q^{d}$ are 
two dense subgroups of $\Q^{d}$, such that
the $\Z^{d}$-odometers, 
$(Y_{H}, \psi_{H})$ and  
$(Y_{H'}, \psi_{H'})$ are isomorphic, then 
there is  $\alpha$ in $GL_{d}(\Z)$  such that 
$(Y_{\alpha H}, \psi_{\alpha H})$ and
 $(Y_{H'}, \psi_{H'})$ are conjugate.
\end{prop}

\begin{proof}
It is clear that $\alpha$ induces automorphisms of both $\Z^{d}$ and 
$\Q^{d}$. It is then a simple matter to observe in our 
definition of $(Y_{H}, \psi_{H})$ that we have a commutative 
diagram

\hspace{3cm}
\xymatrix{\Z^{d} \ar^{\hat{\rho}_{H}}[r] \ar_{\alpha}[d] &
   \widehat{H/ \Z^{d}} \ar^{\alpha}[d]  \\
   \Z^{d} \ar^{\hat{\rho}_{\alpha H}}[r]  &
   \widehat{\alpha H/ \Z^{d} }  } 
   
In fact, the commutativity of this diagram is simply a
re-phrasing of the  desired isomorphism between the two systems,   
implemented by $\alpha$.

For the second part, the isomorphism between the actions 
 provides an 
automorphism of the group $\Z^{d}$. But such 
an automorphism is always implemented by a matrix $\alpha$ as above.
From the first part, we know that 
$(Y_{H}, \psi_{H})$ and $(Y_{\alpha H}, \psi_{\alpha H})$ are 
isomorphic via $\alpha$. It follows that 
the latter is conjugate to $(Y_{H'}, \psi_{H'})$.
\end{proof}

The next result shows that the superindex can be computed from 
a given expression of $H$ as a union of finite index extensions 
of $\Z^{d}$. This is useful, in view of 
Theorems \ref{odom:17} and \ref{odom:20}. The proof
 is trivial and we omit it.

\begin{prop}
\label{odom:30}
Let $ \Z^{d} \subseteq H \subseteq \Q^{d}$.
If $H_{n}, n \geq 1$ is an increasing sequence 
of finite-index extensions of $\Z^{d}$ with union $H$, then 
\[
 [[H: \Z^{d}]] 
  = \cup_{n =1}^{\infty} \{ k \in \N \mid k | [H_{n}: \Z^{d}] \}.
  \]
\end{prop}

\section{Cohomology for $\Z^{d}$-actions}

In this section, we provide some basic definitions and results
regarding cohomology. We begin with a Cantor minimal $\Z^{d}$-system, 
$(X, \varphi)$. We let $C(X, \Z)$ denote the set of 
continuous integer-valued functions on $X$. Of course, each such 
function is simply the (finite) sum of integer multiples 
of characteristic functions of clopen subsets of $X$.
We regard it as an abelian group with pointwise addition of functions.
We note that the non-negative functions form a positive cone.
It is also a $\Z^{d}$-module via $n \cdot f(x) = f( \varphi^{n}(x))$, 
for $n$ in $\Z^{d}$, $f$ in $C(X, \Z)$ and $x$ in $X$.
We define $H^{*}(X, \varphi)$ to be the group 
cohomology of 
$\Z^{d}$ with coefficients in the module $C(X, \Z)$.
This was first considered by Forrest and Hunton \cite{FH}. 

We refer the reader to \cite{Br} for a more thorough 
treatment of cohomology. We remark that this may be described in the 
following fashion. For $n \geq 0$, let $C^{n}$ be the group 
of integer-valued functions on $ X \times \times_{i=1}^{n} \Z^{d}$
which are continuous in the product topology.
We have a coboundary operator $d: C^{n} \rightarrow C^{n+1}$ defined
by
\begin{eqnarray*}
d(\theta)(x, s_{0},s_{1}, \ldots, s_{n})
  &  = &   \theta(\varphi^{s_{0}}(x), s_{1}, \ldots, s_{n}) \\
   &  &  + \sum_{i=1}^{n} (-1)^{i} 
   \theta(x, s_{0},s_{1}, \ldots,s_{i-1}+ s_{i}, \ldots, s_{n}) \\
   &  &   + (-1)^{n+1} \theta(x, s_{0}, \ldots, s_{n-1})
  \end{eqnarray*}
  for $\theta$ in $C^{n}$, $x$ in $X$ and 
  $s_{0}, \ldots, s_{n}$ in $\Z^{d}$. We will let 
  $Z^{n}(X, \varphi), B^{n}(X, \varphi)$ and  $H^{n}(X, \varphi)$ denote
  the $n$-cocycles, $n$-coboundaries and 
  $n$-cohomology groups, respectively,  of this complex.
  
  We will have particular interest in the group $H^{1}(X, \varphi)$. 
  Notice that here we are looking at continuous functions 
  $\theta: X \times \Z^{d} \rightarrow \Z$ and such a function
  is a $1$-cocycle (i.e. $d(\theta)= 0$) if and only if
  \[
  \theta(x, m+n) = \theta(x, m) + \theta(\varphi^{m}(x), n),
  \]
  for all $m,n $ in $\Z^{d}$ and $x$ in $X$.
  A cocycle, $\theta$, is a coboundary if there is $h$ in $C(X, \Z)$ such that 
  $\theta(x, n) = h(\varphi^{n}(x)) - h(x)$, for all 
  $x$ in $X$ and $n$ in $\Z^{d}$.

  \begin{prop}
  \label{coho:2}
  The cohomology $H^{*}(X, \varphi)$ is an invariant 
  of continuous orbit equivalence.
  \end{prop}
  
 We will not prove this.
  One rather long method of proof is by direct computation.
  Another is to observe that the cohomology is actually 
  the groupoid cohomology of the \'{e}tale groupoid 
  $X \times \Z^{d}$ (see \cite{Li})  and that continuous orbit
  equivalence implies (or is actually equivalent to) 
  isomorphism between the \'{e}tale groupoids.
  It is worth noting that the cohomology is 
  \emph{not} an invariant of orbit  equivalence because
  the cocycles are required to be continuous.
  
  The fact that our cohomology groups are coming from dynamical systems 
  provides extra tools for their study. 
  Specifically, our systems always have invariant measures 
  (unique invariant measures for odometers) and these can be paired 
  with cocycles.

  \begin{prop}
  \label{coho:5}
  Let $\mu$ be an invariant probability measure for 
  the  Cantor $\Z^{d}$-system  $(X, \varphi)$. For any 
  $1$-cocycle  define 
  $\tau_{\mu}^{1}(\theta): \Z^{d} \rightarrow \R$ by 
  \[
  \tau_{\mu}^{1}(\theta)(n) = \int_{X} \theta(x, n) d\mu(x),
  \]
  for $n$ in $\Z^{d}$. Then $\tau_{\mu}^{1}(\theta)$
  is a group homomorphism. Moreover, it 
  is zero if $\theta$ is a coboundary and hence passes to 
  a well-defined group homomorphism
  \[
  \tau_{\mu}^{1}: H^{1}(X, \varphi) \rightarrow Hom(\Z^{d}, \R).
  \]
  \end{prop}
  
  \begin{proof}
  First, we check that 
  $\tau_{\mu}^{1}(\theta)$ is a group homomorphism. 
  Using the invariance of $\mu$, 
  for $m,n$ in $\Z^{d}$, we have
  \begin{eqnarray*}
  \tau_{\mu}^{1}(\theta)(m+n) & = & 
   \int_{X} \theta(x, m + n) d\mu(x) \\
   & = & 
   \int_{X} \left( \theta(x, m) + \theta(\varphi^{m}(x), n) \right) d\mu(x) \\
     & = & 
   \int_{X} \theta(x, m) d\mu(x)  +   \int_{X}\theta(\varphi^{m}(x), n) d\mu(x) \\
    & = & 
   \int_{X} \theta(x, m) d\mu(x)  +   \int_{X}\theta(x, n) d\mu(x) \\
   & = &  \tau_{\mu}^{1}(\theta)(m) +   \tau_{\mu}^{1}(\theta)(n).
  \end{eqnarray*}
  
  Next, we check that if $\theta = dh$, then $\tau_{\mu}^{1}(\theta) = 0$.
  Let $n$ be in $\Z^{d}$. Again using the invariance of $\mu$, we have 
  \begin{eqnarray*}
  \tau_{\mu}^{1}(\theta)(n) & = & 
   \int_{X} \theta(x, n) d\mu(x) \\
   & = & 
   \int_{X} h(\varphi^{n}(x)) - h(x) d\mu(x) \\
     & = & 
   \int_{X} h(x) d\mu(x) -  \int_{X} h(x) d\mu(x) \\
    &  =  & 0.
  \end{eqnarray*}
  
   The fact that
   $\tau_{\mu}^{1}$ is additive is obvious. 
  \end{proof}

  We want to make one simplification to this result and 
  that concerns the group $Hom(\Z^{d}, \R)$. There is an 
  obvious isomorphism from this group to 
  $\R^{d}$,  taking $\alpha$ in  $Hom(\Z^{d}, \R)$
  to $(\alpha(e_{1}), \alpha(e_{2}), \ldots, \alpha(e_{d}))$ in $\R^{d}$.
  We simply build this into our definition, without 
  changing our notation.
  
  \begin{defn}
  \label{coho:10}
  Let $\mu$ be an invariant probability measure for 
  the  Cantor $\Z^{d}$-system  $(X, \varphi)$. We define 
  $\tau_{\mu}^{1}: H^{1}(X, \varphi) \rightarrow \R^{d}$ by 
  \[
  \tau_{\mu}^{1}([\theta]) = \left( \tau_{\mu}^{1}(\theta)(e_{1}), \ldots, 
  \tau_{\mu}^{1}(\theta)(e_{d}) \right),
  \]
  for any $1$-cocycle $\theta$.
  \end{defn}
 
It is worth noting that this final version of the  invariant
depends on the generators of $\Z^{d}$. In particular, 
isomorphic systems do \emph{not} have the same map.

 We also introduce the group of co-invariants;
  we let $B(X, \varphi)$ denote the subgroup of  $C(X, \Z)$
   generated by  all functions of the form
   $h - h \circ \varphi^{n}$, where $h$ is any element 
   of $C(X, \Z)$ and $n$ is in $\Z^{d}$. We let 
   \[
   D(X, \varphi ) = C(X, \Z) / B(X, \varphi).
   \]
   We let $[f]$ denote the coset of $f \in C(X, \Z)$.
   We also 
   endow it with the positive cone
   \[
   D(X, \varphi)^{+} = \{  [f ] \mid f \geq 0 \}
   \]
   and order unit $[1]$, where $1$ denotes the constant function.
   
   Once again, if $\mu$ is an invariant probability measure for the 
   system $(X, \varphi)$ then the formula
   \[
   \tau_{\mu}([f]) = \int_{X} f(x) d\mu(x),
   \]
   defines a positive group homomorphism from $D(X, \varphi)$ 
   to $\R$.
   
   We also define 
  $ B_{m}(X, \varphi)$ to be the set of all  $ f$  in $ C(X, \Z)$
  such that $ \int_{X} f d\mu = 0$, for
   all $\varphi$-invariant measures  on $X$. It evidently contains 
   $B(X, \varphi)$ and we let $D_{m}(X, \varphi)$ denote the quotient
   with order structure analogous to the before. This is a 
   quotient of $D(X, \varphi)$.
   
   The importance of the ordered group $D_{m}(X, \varphi)$ is that, 
   for minimal free actions of $\Z^{d}$ on the Cantor set, 
   it is a complete  invariant for orbit equivalence \cite{GMPS2}.
   
   We remark here that for minimal, free Cantor
    $\Z^{d}$-systems, $D(X, \varphi)$ is actually isomorphic 
   to $H^{d}(X, \varphi)$, although the latter 
   has no natural order structure. The isomorphism is induced
   by taking an $d$-cocycle $\theta$ to the function
   $f(x) = \theta(x, e_{1}, \ldots, e_{d})$ in $C(X, \Z)$, where
    $e_{1}, \ldots, e_{d}$ is the standard basis for $\R^{d}$.
   We refer the reader to \cite{FH} although we will not use 
   this fact. We also refer the reader to \cite{GPS2}.

\section{Cohomology for $\Z^{d}$-odometers}

The main results of this section describe the cohomology of
a free, minimal $\Z^{d}$-odometer and are
based on two relatively simple results 
on cohomology.

\begin{lemma}
\label{odco:2}
Let 
\[
(X_{1}, \varphi_{1}) \stackrel{\pi_{1}}{\longleftarrow} (X_{2}, \varphi_{2}) \stackrel{\pi_{2}}{\longleftarrow} \cdots
\]
be a system of $\Z^{d}$-actions and let $(X, \varphi)$ be 
their inverse limit. Then, for all $i \geq 0$, we have
\[
H^{i}(X, \varphi) = \lim_{n \rightarrow \infty} 
H^{i}(X_{1}, \varphi_{1}) \stackrel{\pi_{1}^{*}}{\longrightarrow}
 H^{i}(X_{2}, \varphi_{2}) \stackrel{\pi_{2}^{*}}{\longrightarrow}
 \cdots
 \]
 In addition, we have
  \[
D(X, \varphi) = \lim_{n \rightarrow \infty} 
D(X_{1}, \varphi_{1}) \stackrel{\pi_{1}^{*}}{\longrightarrow}
D(X_{2}, \varphi_{2}) \stackrel{\pi_{2}^{*}}{\longrightarrow}
 \cdots
 \]
\end{lemma}

We will not provide a proof, but we refer the reader to 
\cite{Br}. In fact, the reader can easily construct a proof
 himself or herself by starting with the fact that 
 $C(X \times  \Z^{d} \times \ldots \times \Z^{d}, \Z)$ 
 is the inductive limit of 
 \[
 C(X_{1} \times  \Z^{d} \times \ldots \times \Z^{d}, \Z)
 \stackrel{\pi_{1}^{*}}{\rightarrow} 
 C(X_{2} \times  \Z^{d} \times \ldots \times \Z^{d}, \Z)
 \stackrel{\pi_{2}^{*}}{\rightarrow} \cdots
 \]
 
 The second basic result is the following, which is a 
 very simple case of Shapiro's Lemma \cite{Br}. We 
 will sketch a proof,  in part for completeness 
 and in part because we will need to use some aspects 
 of the proof in the next computation. 

\begin{lemma}
\label{odco:5}
Let $d \geq 1$ and let $G$ be a finite 
index subgroup of $\Z^{d}$.
For each $\theta$ in 
$Z^{1}(\Z^{d}/G, \varphi_{G}) $, we define
$\alpha(\theta): G \rightarrow \Z$  by 
\[
\alpha(\theta)(g) = \theta(G, g), g \in G.
\]
Then $\alpha(\theta)$ is a group homomorphism
 and $\alpha$ 
induces an isomorphism from $H^{1}(X_{G}, \varphi_{G}) $
to $Hom(G, \Z)$.
\end{lemma}

\begin{proof}
The fact that $\alpha(\theta)$ is a group homomorphism 
is a trivial consequence of the cocycle condition on $\theta$, 
when restricted to $\{ G \} \times G$.

Second, it is a trivial computation to see that, if 
$f$ is in $C(\Z^{d}/G, \Z)$, then $\alpha(d(f)) = 0$. This implies 
that $\alpha$ descends to a well-defined map on cohomology.

Third, it is an easy matter to see that 
if $\theta$ and $\eta$ are cocycles, then 
$\alpha( \theta + \eta) = \alpha(\theta) + \alpha( \eta)$.

Fourth, suppose that $\theta$ and $\eta$ are cocycles
and $\alpha(\theta ) = \alpha( \eta)$. 
Select $k_{i},  1 \leq i \leq [\Z^{d}: G]$ in $\Z^{d}$,
 one from  each coset of $G$ in $\Z^{d}$.
Define 
$f$ in $C(\Z^{d}/G, \Z)$ by 
\[
f( k_{i} + G) = \theta(G, k_{i}) - \eta(G, k_{i}), 1 \leq i \leq [\Z^{d}: G].
\]
It is  a simple computation (using the cocycle condition)
to prove that $\theta - \eta - d(f) = 0$. This shows that
the map induced by $\alpha$ at the level of cohomology is injective.

Finally, let $\gamma: G \rightarrow \Z$ be a homomorphism.
Let $k_{i},  1 \leq i \leq [\Z^{d}: G]$ in $\Z^{d}$ 
be as above. To define a cocycle $\theta$, it suffices
to pick $ 1 \leq i,j \leq [\Z^{d}: G]$ and $g$ in $G$
and define 
$\theta(k_{i} + G, k_{j} + g)$. Given $i,j,g$, there is a 
unique $ 1 \leq l \leq [\Z^{d}: G]$ and $g'$ in $G$ such that
$k_{i} + k_{j} + g = k_{l} + g'$ and we set 
\[
\theta(k_{i} + G, k_{j} + g) = \gamma(g').
\]
It is a fairly simple matter to check that 
$\theta$ is a cocycle and it is obvious that 
$\alpha(\theta) = \gamma$.
\end{proof}

\begin{thm}
\label{odco:7}
Let $\Z^{d} \subseteq H \subseteq \Q^{d}$.
Then we have $H^{1}(Y_{H}, \psi_{H}) \cong H$.
\end{thm}

\begin{proof}
Select an increasing  sequence of finite-index extensions of $\Z^{d}$, 
\newline 
$H_{1}, H_{2}, \ldots$ with union $H$.
In view of Theorem \ref{odom:17}, it suffices for
 us to prove that 
$H \cong H^{1}(X_{\mathcal{G}}, \varphi_{\mathcal{G}})$,
 where 
$\mathcal{G} = \{ H_{1}^{*}, H_{2}^{*}, \ldots \}$.
Then we have a commutative diagram

\xymatrix{
H^{1}(\Z^{d}/H_{1}^{*}, \varphi_{H_{1}^{*}}) \ar[d]^{\alpha} \ar[r]
  &  H^{1}(\Z^{d}/H_{2}^{*}, \varphi_{H_{2}^{*}}) \ar[d]^{\alpha} \ar[r] 
     &  H^{1}(\Z^{d}/H_{3}^{*}, \varphi_{H_{3}^{*}}) \ar[d]^{\alpha} \ar[r]   & \cdots \\
Hom(H_{1}^{*}, \Z ) \ar[r] \ar[d] & 
Hom(H_{2}^{*}, \Z ) \ar[r] \ar[d]  & 
Hom(H_{3}^{*}, \Z ) \ar[r] \ar[d] & \cdots \\
H_{1} \ar[r]  & H_{2} \ar[r]  & H_{3} \ar[r] & \cdots }.

Each of the upper vertical maps is an isomorphism by 
Lemma \ref{odco:5}. Each of the lower vertical maps
is an isomorphism by simple duality. The 
limit of the first line is 
$H^{1}(X_{\mathcal{G}}, \varphi_{\mathcal{G}})$ by
 Lemma \ref{odco:2}, 
while the limit of the last line is $H$.
\end{proof}

While the next result undoubtedly holds for
all $d$, the proof is geometric and rather simpler
if we restrict to $d=1,2$. In fact, we will give the
proof only for $d=2$ sure that, having seen this, readers 
can easily supply the proof for $d=1$.

\begin{thm}
\label{odco:8}
Let $d = 1$ or $d=2$ and
let $\Z^{d} \subseteq H \subseteq \Q^{d}$.
Let $\mu$ be the unique invariant probability measure for the system
$(Y_{H}, \psi_{H})$.
Then the  map
\[
\tau_{\mu}^{1}: H^{1}(Y_{H}, \psi_{H}) 
\rightarrow H
\]
is an isomorphism.
\end{thm}

\begin{proof}
Fix a sequence $H_{1}, H_{2} \ldots$ of finite-index 
extensions of $\Z^{2}$ with union  $H$.
Let $ n \geq 1$ and $h$ be in $H_{n}$. Let $G_{n} = H_{n}^{*}$ 
so that $H_{n} = Hom(G_{n}, \Z)$ via 
the inner product. Referring to the commutative 
diagram in the proof of Theorem \ref{odco:7}, 
$h$ determines the homomorphism $< \cdot , h>$ in
$Hom(G_{n}, \Z)$ which in turn determines a cocycle 
$\theta$ (unique up to coboundary)
in  $Z^{1}( \Z^{d}/ G_{n}, \varphi_{G_{n}})$ with 
$\alpha(\theta) = < \cdot , h>$. This means that
\[
\theta( G_{n}, g ) = \alpha(\theta)(g) =  < g,h>,
\]
for all $g$ in $G_{n}$. Our first task is to use the proof
of Lemma \ref{odco:5} to write $\theta$ explicitly.

This begins with the choice of 
$k_{i}, 1 \leq i \leq [\Z^{2}: G_{n}]$, which 
represent the cosets of $G_{n}$.
We may choose generators $(a,b), (c,d)$ of $G_{n}$ 
such that these lie in the first quadrant and 
the line through $(a,b)$ is below the line 
through $(c,d)$; that is,
$a,d > 0$, $b, c \geq 0$ and $ad - bc > 0$. Let 
$k_{i}$ be the points in the integer lattice which are also
in the parallelogram determined by $(a,b)$ and $(c,d)$; 
more precisely, let 
\[
F = \{ s (a,b) + t(c,d) \mid 0 \leq s, t < 1 \} \cap \Z^{2}.
\]
We then define $\theta$ as in the proof of \ref{odco:5}:
 for 
  $k_{1}, k_{2}$ in $F$ and $g$ in $G_{n}$,
  we set $\theta(k_{1} + G_{n}, k_{2} + g) = <g', h>$, 
where $k'$ in $F$ and $g'$ in $G_{n}$ are such that 
$k_{1} + k_{2} + g = k' + g'$.

With a slight abuse of notation, we
 consider the cocycle, again denoted $\theta$, in 
$Z^{1}(X_{\mathcal{G}}, \varphi_{\mathcal{G}})$  given 
by $\theta(x, k) = \theta(\pi_{n}(x), k)$, for
$x$ in $X_{\mathcal{G}}$ and $k$ in $\Z^{2}$.
It is the class of this $\theta$ that is mapped to $h$
under the isomorphism of Theorem  \ref{odco:7}.

It follows from the definition of 
$\tau_{\mu_{\mathcal{G}}}^{1}$ and the formula for
 $\mu_{\mathcal{G}}$ given in Theorem \ref{odom:8}
  that
 \[
 \tau_{\mu_{\mathcal{G}}}^{1}(\theta) 
 =  \sum_{k \in F} [\Z^{d}:G_{n}]^{-1}
  ( \theta(k + G_{n}, (1,0)), \theta(k + G_{n}, (0, 1))  )
  \]
First, we note that $[\Z^{d}:G_{n}] = ad  -bc$.
We compute the first entry of 
$ \tau_{\mu_{\mathcal{G}}}^{1}(\theta) $.
Let $k$ be in $F$. If $k + (1,0) = k'$ is also in $F$, then 
writing $k + (1,0) = k' + (0,0)$ means that 
$\theta(k + G_{n}, (1,0)) = <(0,0), h> = 0$.

Now let us write $k= (i, j)$.
If we fix $0 \leq j < b+d$, the values of 
$i$ for which $(i,j)$ are in $F$ form an interval, 
$i = i_{0}, \ldots, i_{1}$. If $i < i_{1}$, then 
$(i,j) + (1,0)$ is again in $F$ and 
$\theta((i,j) + G_{n}, (1,0)) =  0$. Let us now consider
$i= i_{1}$. If $0 \leq j < b$, then the point 
$(i_{1}, j) + (1,0)$ has moved out of $F$
 through its lower boundary, the line joining
 the origin and $(a,b)$. In this case we write
 $(i_{1}, j) + (1,0) = (i_{1}+1 + c,j+ d) - (c,d)$, 
 where $(i_{1}+1 + c, j+d)$ is in $F$ and $ - (c,d)$
 is in $G_{n}$. Hence, we have 
 \[
 \theta( (i_{1}, j) + G_{n}, (1,0)) = <-(c,d), h>.
 \]
 We note that there are exactly $b$ such 
 values of $(i_{1}, j)$. (This conclusion also holds in the 
 case $b=0$, which we leave to the reader.)
 
 If $ b \leq j < b+d$, then $(i_{1}, j) + (1, 0)$
 has moved out of $F$ through its right boundary and we
 write $(i_{1}, j) + (1,0) = (i_{1}+1 -a, j-b) + (a, b)$,
 with $(i_{1}+1 -a, j-b)$ in $F$ and $ (a, b)$ in $G_{n}$,
 and we have 
 \[
 \theta( (i_{1}, j) + G_{n}, (1,0)) = <(a,b), h>.
 \]
 There are exactly $d$ such values of $(i_{1},j)$.
  Altogether, we find the first entry of 
$  \tau_{\mu_{\mathcal{G}}}^{1}(\theta)$ is
\[
  (ad -bc)^{-1}\left(  b < -(c,d), h>  + d < (a,b), h> \right)
  =  < (1,0), h>.
  \]
  In a similar way, the second entry is $<(0,1), h>$ and so 
  we have 
  \[
   \tau_{\mu_{\mathcal{G}}}^{1}(\theta) = 
   (<(1,0),h>, <(0,1), h>) 
    = h.
    \]
    This completes the proof. 
\end{proof}

\section{Classification of $\Z^{d}$-odometers}

\begin{cor}
\label{class:10}
Let $\Z^{d} \subseteq H, H' \subseteq \Q^{d}$ be dense subgroups
of $\Q^{d}$..
\begin{enumerate}
\item
If $ d= 1$ or $d=2$, then the systems $(Y_{H}, \psi_{H})$ and 
$(Y_{H'}, \psi_{H'})$ are conjugate if and only 
if $H = H'$.
\item 
The systems $(Y_{H}, \psi_{H})$ and 
$(Y_{H'}, \psi_{H'})$ are isomorphic if and only 
if there exists  $\alpha$ in 
$GL_{2}(\Z)$ such that 
$\alpha H = H'$.
\end{enumerate}
\end{cor}

\begin{proof}
The first part is an immediate consequence of Theorem \ref{odom:8}. 
The second part is an immediate consequence of Theorem \ref{odom:25} 
and the first part.
\end{proof}

\begin{ex}
\label{class:11}
We remark that the condition that $\alpha$ be in 
$GL_{2}(\Z)$ cannot be replaced 
by $\alpha$ in $SL_{2}(\Z)$. Consider 
$H = \Z[1/2] \oplus \Z[1/3] + \Z(1/5, 1/5)$ and $H' = \alpha H$, where 
$\alpha = \left[ \begin{array}{cc} 0 & 1 \\ 1 & 0 \end{array} \right]$.
We claim that if $\beta$ is any matrix in $GL_{2}(\Z)$ such that 
$\beta H = \alpha H$, then $\beta = \pm \alpha$ and hence there 
is no such $\beta$ in $SL_{2}(\Z)$. To see the claim, let
$K$ be the subgroup of $H$ consisting of all elements $h$ such that, for 
every $k \geq 1$, there is $h'$ in $G$ such that $2^{k}h' = h$. 
It is easy to see that $K = \Z[1/2] \oplus 0$. 
Similarly, we let $L$ be an analogous group, replacing $2$ by $3$, so 
that $L = 0 \oplus \Z[1/3]$. It is then clear that the only
$\beta$ in $GL_{2}(\Z)$ such that $\beta K = \alpha K$ and 
$\beta L = \alpha L$ are  
$\beta = \left[ \begin{array}{cc} 0 & \pm 1 \\ \pm 1 & 0 \end{array} \right]$. 
As we also require, $\beta (1/5, 1/5)$ to be in $\alpha H$, this leaves only
$\beta = \pm \alpha$.
\end{ex}

We now turn our attention to orbit equivalence
for $\Z^{d}$-odometers. 
The following result is also trivial; we state it simply
for emphasis.

\begin{lemma}
\label{class:13}
Let $G$ be a finite index subgroup 
of $\Z^{d}$. Let $\mu$ be the normalized counting measure on
$\Z^{d}/G$. Then 
$B(\Z^{d}/G, \varphi_{G}) = B_{m}(\Z^{d}/G, \varphi_{G})$
and 
\[
\tau_{\mu}: D(\Z^{d}/G, \varphi_{G}) = D_{m}(\Z^{d}/G, \varphi_{G}) 
\rightarrow [\Z^{d}: G] ^{-1} \Z
\]
is an isomorphism of ordered abelian groups with order units.
\end{lemma}

\begin{thm}
\label{class:14}
Let $\Z^{d} \subseteq H \subseteq \Q^{d}$.
 Then 
$B(Y_{H}, \psi_{H}) = B_{m}(Y_{H}, \psi_{H})$
and  the map
\[
\tau_{\mu}^{d}: D(Y_{H}, \psi_{H}) 
\rightarrow  \cup_{m \in [[H: \Z^{d}]]} \, m^{-1} \Z
\]
is an isomorphism of ordered abelian groups with order units.
\end{thm}

\begin{proof}
The proof is exactly the same as that of 
 Theorem \ref{odco:7}, with Lemma \ref{class:13} replacing 
  Lemma \ref{odco:5}. We omit the details.
\end{proof}

\begin{cor}
\label{class:16}
Let $\Z^{d} \subseteq H \subseteq \Q^{d}$ and 
 $\Z^{d'} \subseteq H' \subseteq \Q^{d'}$ be dense subgroups.
The systems $(Y_{H}, \psi_{H})$ and 
$(Y_{H'}, \psi_{H'})$ are orbit equivalent if and only \newline
if $[[H: \Z^{d}]] = [[H': \Z^{d'}]] $.
\end{cor}

Notice that the condition above does not require $d = d'$.

We finally turn our attention to the issue of 
continuous orbit equivalence. Here, our main result
is Theorem \ref{class:25} below.  In fact, it is not difficult
to obtain Theorem \ref{class:25} from the results 
of \cite{CM}, but we give an independent direct proof. Even if 
the terminology is different, the proofs share many features.
We also direct the reader's attention 
to Theorem 1.2 of \cite{Li} where several other characterizations 
of continuous orbit equivalence are given.

The following preliminary result will be useful in the proof
and possibly of some interest on its own.

\begin{prop}
\label{class:20}
Assume that $K, H$ are groups with 
$\Z^{d} \subseteq K \subseteq H \subseteq \Q^{d}$, $[K:\Z^{d}]$ is finite 
 and that $H$ 
is dense in $\Q^{d}$. Let $\pi$ denote the quotient map
from $\widehat{H/ \Z^{d}}=Y_{H}$ to $\widehat{K / \Z^{d}} = Y_{K}$, 
$\mathcal{P}$ be the partition  of $Y_{H}$ into clopen sets 
formed by the pre-images of
the points of $Y_{K}$ under $\pi$ and $Z$ be the element of $\mathcal{P}$ 
containing the identity element of $Y_{H}$.
\begin{enumerate}
\item 
If $a_{1}, a_{2}, \ldots, a_{k}$ are representatives of the cosets 
of $K^{*}$ in $\Z^{d}$ (so that $k = [K: \Z^{d}]$), 
then $\mathcal{P}$ consists of the sets 
$\psi^{a_{i}}_{H}(Z), i = 1, \ldots, k$.
\item 
For any $z$ in $Z$ and $n$ in $\Z^{d}$, $\psi_{H}^{n}(z)$ is in 
$Z$ if and only if $n $ is in $K^{*}$. Moreover, 
$\{ \psi_{H}^{n}(z) \mid n \in K^{*} \}$ is dense in $Z$.
\item
If we choose an integer matrix $\alpha$ with $\alpha^{T} \Z^{d} = K^{*}$, then  
the system $(Z, \psi_{H}, K^{*} )$ is isomorphic to
$(Y_{\alpha H}, \psi_{\alpha H}, \Z^{d})$.
\end{enumerate}
\end{prop}

\begin{proof}
Consider the following commutative diagram: 

\hspace{1cm}
 \xymatrix{0 \ar[r] &  K/ \Z^{d} \ar^{i}[r] \ar[d] & H / \Z^{d} \ar^{p}[r] \ar[d] & 
                       H / K \ar[r] \ar[d] & 0 \\
 0 \ar[r] &  K/ \Z^{d} \ar[r]  & \R^{d} / \Z^{d} \ar^{q}[r]  & 
                       \R^{d} / K \ar[r]  & 0      }

If we take Pontryagin duals throughout, the first line gives an exact sequence with $\hat{i} = \pi$.
 It is straightforward to calculate that, after identifying the dual of $\R^{d}/ \Z^{d}$ with $\Z^{d}$, the 
 image of $\hat{q}$ is exactly $K^{*}$. 
 The first part follows at once, as does the first sentence of part 2.
 For the second sentence of part 2, we know that the orbit of $z$ is dense in $Y_{H}$. On the other hand, 
 points of the form $\psi_{H}^{n}(z)$ with $z \notin K^{*}$ 
 do not lie in $Z$ and as $Z$ is clopen, such points cannot 
 limit on points in $Z$. 
 
 For the last statement, we first note that $\alpha^{T} \Z^{d} = K^{*}$ implies that 
 $K = \alpha^{-1} \Z^{d}$. We consider the following commutative diagram
 
 \hspace{3cm}
 \xymatrix{  H / K \ar^{\alpha}[r] \ar[d] & \alpha H / \Z^{d} \ar[d] \\
                       \R^{d} / K \ar^{\alpha}[r]  &     \R^{d} /  \Z^{d}   }
                       
Of course, the horizontal maps are isomorphisms. Taking Pontryagin duals yields the last part. 
\end{proof}

\begin{thm}
\label{class:25}
Let $\Z^{d} \subseteq H \subseteq \Q^{d}$ and 
$\Z^{d'} \subseteq H' \subseteq \Q^{d'}$ be dense subgroups.
If 
$d=d'$ and there is $\alpha$ in $GL_{d}(\Q)$ with $det(\alpha) = \pm 1$
 such that 
$\alpha H = H'$, then the odometers $(Y_{H}, \psi_{H})$ and $(Y_{H'}, \psi_{H'})$
are continuously orbit equivalent. The converse holds in the case $d=1,2$. 
\end{thm}

\begin{proof}
First, let us assume that the systems are continuously orbit 
equivalent. We have noted in Proposition \ref{coho:2} that continuous orbit 
equivalence implies isomorphism of cohomology groups, that
$H^{d}(Y_{H}, \psi_{H}) \cong D(Y_{H}, \psi_{H})$ is non-trivial
and that $H^{k}(X, \varphi)=0$, for any $k > d$ and any 
$\Z^{d}$-action. We conclude from this that $d=d'$, being the largest 
integer with non-trivial cohomology in that degree.

Now let $h: Y_{H} \rightarrow Y_{H'}$ be the homeomorphism
and $\alpha: H \times \Z^{d} \rightarrow \Z^{d}$ be the associated cocycle
as in the definition of continuous orbit equivalence.
By following $h$ by rotation by $-h(0)$, we may assume that $h(0) = 0$
in the group $Y_{H'}$.
Also let $H_{n}, n \geq 1$ be an increasing sequence
of finite index extensions of $\Z^{d}$ with union $H$.  
This means that $Y_{H}$ is (up to isomorphism) an inverse limit of
finite spaces, $\Z^{d} / H_{n}^{*}$. Recall that we let
$\pi_{n}$ denote the map from $Y_{H}$ to $\Z^{d}/H_{n}^{*}$.
For each $i = 1, 2, \ldots, d$, the function $\alpha( \cdot, e_{i})$
is continuous and takes values in $\Z^{d}$. Hence it is locally
 constant and 
we may find $n$ such that each function 
$\alpha( \cdot, e_{i}), i = 1, \ldots, d$
is constant on 
the partition induced by $\pi_{n}$. 

For convenience, we let $K = H_{n}$, $\pi: Y_{H} \rightarrow Y_{K}$
be the quotient map and the partition $\mathcal{P}$ be as in 
Proposition \ref{class:20}. 
It follows from the cocycle property of $\alpha$ and the fact that 
$\psi_{H}$ simply permutes the elements of $\mathcal{P}$ that 
$\alpha( \cdot , m)$ is constant on each element
of the partition for every $m$ in $\Z^{d}$. 
Define $\alpha: K^{*} \rightarrow \Z^{d}$ by 
$\alpha(n) = \alpha(0_{H}, n), n \in K^{*}$. 
It is then clear from the cocycle condition and our choice of $Z$ 
in Proposition \ref{class:20} that 
$ \alpha(K^{*})$ is a subgroup of $\Z^{d}$.

It follows from Proposition
\ref{class:20} that 
\[
K^{*} = \{ k \in \Z^{d} \mid \psi_{H}^{k}(0_{H}) \in
 Z \}.
 \]
 Then as $h$ is an orbit equivalence, we have 
 \[
 \alpha (K^{*}) =  \{ l \in \Z^{d} \mid \psi_{H'}^{l}(0_{H'}) \in
h( Z) \}.
\]

We claim that $K' = (\alpha(K^{*})^{*}$ is actually a subgroup of $H'$.
Now let $H'_{m}$ be an increasing sequence of finite index extensions 
of $\Z^{d}$ with union $H'$ and let $\pi'_{m}$ be the natural 
map from $Y_{H'}$ to $Y_{H'_{m}}$.
The clopen sets $(\pi'_{m})^{-1} \{ 0_{H'_{m}} \}, m \geq 1$ 
form a decreasing system
of sets which intersect to the identity in $Y_{H'}$. It follows
that there exists some $m$ such that  $\pi'_{m}( H'_{m})$ 
is contained in 
$h(Z)$. 
With another application of Proposition
\ref{class:20}, we have 
\begin{eqnarray*}
H_{m}'^{*}  &  =  &  \left\{ l \in \Z^{d} \mid 
\psi_{H'}^{l}(0_{H'}) \in \pi'^{-1}\{ 0_{H_{m}'} \}  \right\} \\
  &  \subseteq & \left\{ l \in \Z^{d} \mid \psi_{H'}^{l}(0_{H'}) \in
h( \pi^{-1}\{ 0_{H_{m}'} \})  \right\} \\
   &  =  &  \alpha(K^{*}).
   \end{eqnarray*}
   It follows that 
   \[
   H' \supseteq H'_{m}  \supseteq (\alpha K^{*})^{*} = K'.
   \]
   as desired. For convenience, let  $\pi': Y_{H'} \rightarrow Y_{K'}$.
   Let $Z' $ be the pre-image in $Y_{H'}$ of the identity of $Y_{K'}$ 
   under $\pi'$.
   
   It follows from two applications of the second part of 
   Proposition \ref{class:20} that  
$   h(Z) = Z'$
   and that $h$ with $\alpha$ are an isomorphism 
   between the systems $(Z, K^{*}, \psi_{H})$
   and $(Z', K'^{*}, \psi_{H'})$.
   In addition, $h$ bijectively 
   maps the partition of $Y_{H}$ into 
   the elements of the partition induced by $\pi'$
   of $Y_{H'}$. From this we conclude that each partition 
   has the same number of elements
    and so $[\Z^{d}: (K')^{*}] = [\Z^{d}: K^{*} ]$.
    
    Now choose integer matrices $\alpha$ and $\alpha'$ such that 
    $K^{*} = \alpha^{T} \Z^{d}$ and $(K')^{*} = (\alpha')^{T} \Z^{d}$.  
   We have
    $\vert det(\alpha)\vert = [K:\Z^{d}] = 
     [K':\Z^{d}] = \vert det(\alpha') \vert$. 
   It follows from what we have above and the last part of \ref{class:20} that 
   $(Y_{\alpha H}, \psi_{\alpha H})$ and 
   $(Y_{\alpha' H'}, \psi_{\alpha' H'})$ are isomorphic.
   Therefore, by \ref{class:10} that there is $\beta$ in $GL_{d}(\Z)$ with 
   $\beta \alpha H = \alpha' H'$. Hence, 
   $(\alpha')^{-1} \beta \alpha H = H'$ is a
    matrix with rational entries and    
   \[
   \vert det((\alpha')^{-1} \beta \alpha)\vert =
    \vert det(\alpha')^{-1}\vert
     \vert det(\beta) \vert \vert det (\alpha) \vert
    = 1.
    \]
    
For the converse, suppose that there is a
 matrix $\alpha$  with rational entries and determinant $\pm 1$ such 
that $\alpha H = H'$.  The group $H$ is the
 union of subgroups $K$ with $[K:\Z^{d}]$ finite. 
For each element of $H'$, we may find a $K$ such 
that $\alpha K$ contains that element. As $\Z^{d}$ is finitely generated,
we may find $\Z^{d} \subseteq K \subseteq H$ with $[K: \Z^{d}]$ finite and 
$\Z^{d} \subseteq \alpha K $. Find a positive
 integer $n$ such that $n \alpha$ has only integer entries.
Then we have 
\begin{eqnarray*}
[K: \Z^{d} ] & = & [\Z^{d}: K^{*} ]\\
    & = &   [\Z^{d}: n K^{*} ] [ K^{*}: nK^{*} ]^{-1}\\
   &  =  & [ \Z^{d}: n \alpha^{T} K'^{*}] n^{-d} \\
   & = & \vert det(n \alpha) \vert  [\Z^{d}:  K'^{*} ] n^{-d} \\
    & = &   [\Z^{d}:  K'^{*} ] \\
      &  =  & [K': \Z^{d}] 
    \end{eqnarray*}
    
 Using the fact above, we let
  $a_{1}, a_{2}, \ldots, a_{m}$  and
   $a'_{1}, a'_{2}, \ldots, a'_{m}$  be distinct representatives of the cosets 
 of $K^{*}$  and $K'$, respectively, in $\Z^{d}$.
The map $\alpha$ evidently induces a homeomorphism
 which we denote by $h$ between $Z = \widehat{H/K}$ and 
$\widehat{H'/K'}$, which together with $(\alpha^{T})^{-1}$, 
provide an isomorphism between
$(Z, \psi_{H}, K^{*})$ and $(Z', \psi_{H'}, K'^{*})$.
 We extend $h$ to all of $Y_{H}$ by setting
\newline
$h \mid \psi_{H}^{a_{i}}(Z) = 
\psi_{H'}^{a'_{i}} \circ h \circ \psi_{H}^{-a_{i}}$. 
It follows from 
Corollary \ref{class:10} that 
$h$ is an orbit equivalence between 
$(Y_{H}, \psi_{H})$ and $(Y_{H'}, \psi_{H'})$ and it is an easy matter to see
that the associated cocycles are continuous. We omit the details.
\end{proof}

\begin{rem}
\label{class:26}
We leave it as an exercise to show that, for $H$ and $H'$ as in 
Example \ref{class:11}, the only $\beta $ in $GL_{2}(\Q)$ with 
such that 
$\beta H = H'$ and $\det(\beta) = \pm 1$ is $\beta = \pm \alpha$. 
This implies that we cannot
change the condition $\det(\alpha) = \pm 1$ to 
$\det(\alpha) =1$ above.
\end{rem}

These results on classification lead to a 
surprising dichotomy between the cases $d = 1$ and $d = 2$,
which essentially stems from the fact that if a subgroup
of the rationals
is a finite index extension of the integers, then
 the index uniquely 
determines the group.

\begin{cor}
\label{class:30}
Let $\Z \subseteq H, H' \subseteq \Q$ be dense. The following are equivalent.
\begin{enumerate}
\item  The $\Z$-odometers $(Y_{H}, \psi_{H})$ and $(Y_{H'}, \psi_{H'})$ 
are conjugate.
\item The $\Z$-odometers $(Y_{H}, \psi_{H})$ and $(Y_{H'}, \psi_{H'})$ 
are  isomorphic.
\item The $\Z$-odometers $(Y_{H}, \psi_{H})$ and $(Y_{H'}, \psi_{H'})$ 
are continuously orbit equivalent.
\item The $\Z$-odometers $(Y_{H}, \psi_{H})$ and $(Y_{H'}, \psi_{H'})$ 
are orbit equivalent.
\end{enumerate}
\end{cor}

\begin{proof}
This is well-known, but we will give a 
proof here since it is quite short. The implications 
$(1) \Rightarrow (2) \Rightarrow (3) \Rightarrow (4)$ are all 
trivial. We must prove $(4) \Rightarrow (1)$.
From \ref{class:10}, we know that $\tau^{1}(H^{1}(Y_{H}, \psi_{H}))$
is a complete invariant for conjugacy, while
 from \cite{GPS1} and Theorem \ref{class:14}, 
$\tau^{d}(D(Y_{h}, \psi))$ is a complete invariant for orbit equivalence. 
As we stated earlier, 
$D(Y_{H}, \psi_{H}) = H^{1}(Y_{H}, \psi_{H})$. This completes the proof.
 \end{proof}

Let us take this opportunity to make some 
 vague comments 
 prompted by the last 
 corollary. The first, very briefly, is that  for 
 $\Z^{d}$-odometers,  if $H_{n}, n \geq 1,$ 
 is an increasing sequence of finite-index extensions
 of $\Z^{d}$ with union $H$,  the invariant  
 $H^{1}(Y_{H}, \psi_{H})$ 'remembers' the 
 \emph{groups}
 $H_{n}/\Z^{d}$, whereas 
 the invariant  
 $H^{d}(Y_{H}, \psi_{H})$ 'remembers' the 
 \emph{sets}
 $H_{n}/\Z^{d}$. Secondly, it is well-known 
 that for general minimal, free $\Z^{d}$-actions
 on Cantor sets, the invariant 
 $H^{d}(X, \varphi)$ provides the 
 invariant for orbit equivalence. On the other 
 hand, the role of $H^{1}(X, \varphi)$ is not 
 well-understood. The fact that they coincide when $d=1$
 tends to confuse the issue, which is nicely illustrated in the 
 last corollary.
 
The situation is remarkably different for $d=2$.
We list these examples 
for completeness, but leave the reader the easy
task of verifying them. The first example is already
 in the work of Cortez \cite{Co}. The second and third 
 already appear in Li \cite{Li}. 
 The fourth example is already well-known since \cite{GMPS2}.
 
\begin{ex}
\label{class:35}
\begin{enumerate}
\item Let $H = \Z[1/2] \oplus \Z[1/3]$ and 
 $H' = \Z[1/3] \oplus \Z[1/2]$. Then the 
 $\Z^{2}$-odometers
 $(Y_{H}, \psi_{H}) $ and  $(Y_{H'}, \psi_{H'}) $ are 
 isomorphic (using  $\alpha = \left[ \begin{array}{cc} 0 & 1 \\ 1 & 0 \end{array} \right]$), but not conjugate.
 \item Let $H = \Z[1/2] \oplus 5^{-1}\Z[1/3]$ and 
 $H' = 5^{-1}\Z[1/2] \oplus \Z[1/3]$. Then the 
 $\Z^{2}$-odometers
 $(Y_{H}, \psi_{H}) $ and  $(Y_{H'}, \psi_{H'}) $ are 
 continuously orbit equivalent  (using 
  $\alpha = \left[ \begin{array}{cc} \frac{1}{5} & 0 \\ 0 & 5 \end{array} \right]$), but not isomorphic.
 \item Let $H = \Z[1/2] \oplus \Z[1/15]$ and 
 $H' = \Z[1/10] \oplus \Z[1/3]$. Then the 
 $\Z^{2}$-odometers
 $(Y_{H}, \psi_{H}) $ and  $(Y_{H'}, \psi_{H'}) $ are 
 orbit equivalent, but not continuously orbit equivalent.
 \item Let $\Z^{d} \subseteq H \subseteq \Q^{d}$ be any 
 dense subgroup with $d > 1$. Let 
\[
 H'  = \cup_{m \in [[H: \Z^{d} ]]} m^{-1} \Z,
 \]
 so that $\Z \subseteq H' \subseteq \Q$ is dense. 
 Then $(Y_{H}, \psi_{H}) $ and  $(Y_{H'}, \psi_{H'}) $ are 
 orbit equivalent, but not continuously orbit equivalent.
\end{enumerate}
\end{ex}

One slightly unfortunate consequence of these examples 
 is that 
it may leave the reader with the impression that dense subgroups of
$\Q^{2}$ are mainly obtained by taking direct sums. To allay that, let
us first mention an example due to Fuchs \cite{Fu}
 of a group (already appearing in
\ref{class:11})
$\Z^{2} \subseteq H \subseteq \Q^{2}$ which cannot be written as
an internal direct product in a non-trivial way; that is, the group is 
indecomposable:
\[
H = 
\Z[1/2] \oplus \Z[1/3] + \Z(1/5, 1/5).
\]

Unfortunately, this may still leave the reader with the impression that 
dense subgroups of $\Q^{2}$ are direct sums, up to a finite-index subgroup.
We present an example of a class which seem to be substantially further
from direct sums. 
We have not been able to find in the literature an example of this type.

\begin{prop}
\label{}
There exists a group $\Z^{2} \subseteq H \subseteq \Q^{2}$ such that, 
$H$ is dense in $\Q^{2}$ and,
for every $x$ in $\Z^{2}$, the group
$\Q x \cap H $ is cyclic.
\end{prop}

\begin{proof}

We will consider a pair of positive integers $a, b \geq 2$ and the
 associated matrix
\[
\alpha = \left[ \begin{array}{cc} a & 1 \\ 1 & b \end{array} \right].
\]
Observe that the determinant of $\alpha$ is $ \vert \alpha \vert = ab-1 \neq 0$.
First, we observe that, given any positive integer $K \geq 3$, there 
are $a, b \geq K$
 such 
that $\vert \alpha \vert $ is prime. If we simply let $a=K$, the arithmetic 
progression $bK -1$ is prime for infinitely many $b$ since 
$K$ and $1$ are relatively prime \cite{Di}.

Writing elements of $\Z^{2}$ as row vectors, observe the following, which 
we leave as exercises:
\begin{enumerate}
\item $\Z^{2} \alpha \subseteq \Z^{2}$ and hence $\Z^{2} \subseteq \Z^{2} \alpha^{-1}$,
\item  $\pm \vert \alpha \vert^{-1}( a, -1),
 \pm \vert \alpha \vert^{-1} (-1, b)$
  are  in $\Z^{2}\alpha^{-1}$. 
\item  $\left( \Q \oplus 0 \right) \cap \Z^{2}\alpha^{-1}  = \Z \oplus 0$ and 
 $\left( 0 \oplus \Q \right) \cap \Z^{2}\alpha^{-1}  = 0 \oplus \Z$. 
  
  \item for $ 1 \leq k \leq a-1$, $ b \leq kb \leq (ab-1) - b +1$ and 
  for $ 1 \leq k \leq b-1$, $a \leq ka \leq (ab-1) - a +1$,
  \item for $x\neq 0$ in $ \Z^{2} \alpha^{-1}$, we have 
  $  \vert x \vert_{1} \geq \vert \alpha \vert^{-1} \min \{ a, b \}$,
  where $\vert \cdot \vert_{1}$ is the $\ell^{1}$-norm.
  \end{enumerate}
    
 We claim that $\alpha$ has the following property:
 if $x$ is in $\Z^{2}\alpha^{-1} - \Z^{2}$ and $m > 1$ is such that 
 $mx$ is in $\Z^{2}$, then $\vert mx \vert_{1} \geq K$. We know that 
 \[
 [\Z^{2} \alpha^{-1} : \Z^{2}] = [ \Z^{2}: \Z^{2}\alpha ] = \vert \alpha \vert
 \]
 is prime and so the element $x + \Z^{2}$ must have order $\vert \alpha \vert$ in the 
 quotient group $\Z^{2}\alpha^{-1}/ \Z^{2}$. It follows that $ m \geq \vert \alpha \vert$.
 Combining this with the last item above, we have 
 \[
 \vert mx \vert_{1} \geq \vert \alpha \vert \vert x \vert_{1} \geq 
  \min \{ a, b \} \geq   K.
  \]

We now construct a sequence of matrices as above inductively.
Choose $\alpha_{1}$ using the constant $K = 1$. For $n = 2, 3, 4, \ldots$, we choose 
$\alpha_{n}$ for the constant $K_{n} = n \Vert \alpha_{1} \Vert \Vert \alpha_{2} \Vert \ldots 
\Vert \alpha_{n-1} \Vert$, where $\Vert \alpha \Vert$ is the operator norm, regarding 
$\alpha$ as an operator on $\ell^{1}(\R^{2})$.

We define $H_{n} = \Z^{2} \alpha_{n}^{-1} \alpha_{n-1}^{-1} \cdots \alpha_{1}^{-1}
\subseteq   \Z^{2} \alpha_{n-1}^{-1} \cdots \alpha_{1}^{-1} = H_{n-1}$
and  $H = \cup_{n} H_{n}$.     
    
Now suppose that $x$ is in $H_{n} - H_{n-1}$ and $mx$ is in $\Z^{2}$, for some 
$m > 1$.It follows that $z = x\alpha_{1} \cdots \alpha_{n-1}$ is in $\Z^{2}\alpha_{n}^{-1}$ 
and $mz$ is  
in       $\Z^{2}\alpha_{1} \cdots \alpha_{n-1} \subseteq \Z^{2}$. It follows that 
$\vert mz \vert_{1} \geq n  \Vert \alpha_{1} \Vert \Vert \alpha_{2} \Vert \ldots 
\Vert \alpha_{n-1} \Vert$.
On the other hand, we have 
\[
\vert mz \vert_{1} = \vert  mx\alpha_{1} \cdots \alpha_{n-1} \vert_{1} \leq 
\vert mx \vert_{1} \Vert \alpha_{1} \Vert \Vert \alpha_{2} \Vert \ldots 
\Vert \alpha_{n-1} \Vert.
\]
Combining these two estimates yields
\[
\vert mx \vert_{1} \geq n.
\]

Put another way, we have proved that if $y$ is in $\Z^{2}$ and 
$mx = y$ for some $m > 1$ and $x$ in $H_{n} - H_{n-1}$, then $n \geq \vert y \vert_{1}$.
From this it follows that 
\[
H_{n} \cap \Q y = H_{\vert y \vert_{1} } \cap \Q y, 
\]
if $n \geq \vert y \vert_{1}$. From this and the fact that $H_{\vert y \vert_{1}}$ is discrete, the desired conclusion follows.

To see that $H$ is dense, we observe that its closure  (in $\R^{2}$) 
is $(H^{*})^{*}$. If we observe that , for any matrix $\alpha$ as above
and any vector $(x, y)$ in $\R^{2}$, we have 
\begin{eqnarray*}
\Vert (x, y) \alpha \Vert_{1} &  = & \vert ax + y \vert + \vert x + by \vert \\
   & \geq  & 
(a-1) \vert x \vert + (b-1) \vert y \vert  \\
   & \geq &  (K-1) \Vert (x, y) \Vert_{1}.
\end{eqnarray*}
As
$ K_{n}  \geq n$, it follows that the intersection
of the $\ell^{1}$-ball of radius $n$ with 
\[
H_{n}^{*} = (\Z^{2}\alpha_{n}^{-1} \cdots \alpha_{1}^{-1})^{*} = \Z^{2} (\alpha_{1} \alpha_{2} \cdot \alpha_{n})^{T}.
\]
is trivial. It follows that 
\[
H^{*} = \left( \cup_{n} H_{n}\right)^{*} = \cap_{n} H_{n}^{*} = \{ 0 \}.
\]
Hence, we have $\overline{H} = (H^{*})^{*} = \R^{2}$
and this completes the proof.
\end{proof}

\section{Rational subgroups and odometer factors}

In this section, we turn to examine a general minimal, 
free action, $\varphi$, of $\Z^{d}$ 
on a Cantor set, $X$. Our interest will be in factor maps 
from $(X, \varphi)$ to other systems, particularly odometers.

\begin{thm}
\label{rational:3}
Let $(X, \varphi)$ be a minimal, free $\Z^{d}$-Cantor system.
The group $H^{1}(X, \varphi)$ is torsion-free.
\end{thm}

\begin{proof}
Suppose that $n \geq 1$ and we have an element 
of $H^{1}(X, \varphi)$ of order $n$.
This means that there exists a
cocycle $\theta$ in $Z^{1}(X, \varphi)$
 such that 
 $n\theta$ is in $B^{1}(X, \varphi)$. 
 So we may find $f $ in $C(X, \Z)$ 
 such 
 that 
 \[
 n \theta(x, k) = d(f)(x, k) = 
 f(x) - f(\varphi^{k}(x)), x \in X, k \in \Z^{d}.
 \]
 For $i = 0, 1, 2, \ldots, n-1$, define 
 \[
 X_{i} = \{ x \in X \mid f(x) \equiv i (mod \,n ) \}.
 \]
 In the formula above relating $f$ and $\theta$, it is clear 
 that the left hand side is a multiple of $n$, which 
 immediately implies that $\varphi^{k}(X_{i}) = X_{i}$, 
 for all $k$ in $\Z^{d}$. By minimality, all $X_{i}$ 
 are empty, except one, say $X_{j} = X$. 
 It then follows that 
 $n^{-1}(f - j)$ is in $C(X, \Z)$ and $d(n^{-1}(f - j)) = \theta$. 
 Thus, $\theta$ is zero in $H^{1}(X, \varphi)$.
\end{proof}

\begin{defn}
\label{rational:7}
Let $(X, \varphi)$ be a minimal, free $\Z^{d}$-Cantor
 system. We define the \emph{rational subgroup} 
 of $H^{1}(X, \varphi)$, denoted $\Q(H^{1}(X, \varphi))$, 
 by
 \begin{eqnarray*}
 \Q(H^{1}(X, \varphi)) & = &  \{
 a \in H^{1}(X, \varphi) \mid  \text{ there exist }  \\
   &  &  j \in \Z^{d}, k \geq 1, ka = \sum_{i} j_{i}[\varepsilon_{i}] \}.
 \end{eqnarray*}
\end{defn}

It is an easy matter to check that 
$\Q(H^{1}(X, \varphi))$ is a subgroup of 
$H^{1}(X, \varphi)$ and that it contains 
$[\varepsilon_{i}], 1 \leq j \leq d$.

\begin{lemma}
\label{rational:9}
Let $(X, \varphi)$ be a minimal, free $\Z^{d}$-Cantor
 system. 
 Let $f: X \rightarrow \R$ be continuous
 and let $r$ be in $\R^{d}$. 
 Assume that $0$ is in the range of $f$.
 The following are equivalent.
 \begin{enumerate}
 \item $r$ is in $\Q^{d}$ and 
 \[
 e^{ 2 \pi i f } \circ \varphi^{l} =  e^{2 \pi i ( <r, l>  + f)} ,
 \]
 for all $l $ in $\Z^{d}$. 
 \item There is a positive integer $k$ such that 
 $kf$ is in $C(X, \Z)$ and $ kr$ is in $\Z^{d}$.
 In addition, the function
 \[
 \theta(x,l) = d(f)(x,l) - < r, l>, x \in X, l \in \Z^{d},
 \]
 is in $Z^{1}(X, \varphi)$. Moreover, we have 
  $k[\theta] = \sum_{i=1}^{d} - k r_{i} [\varepsilon_{i}]$ 
  and so $[\theta]$ is in $\Q(H^{1}(X, \varphi))$ and 
  $\tau_{\mu}^{1}[\theta] = r$.
 \end{enumerate}
\end{lemma}

\begin{proof}
First, assume that condition 1 holds. Let $F$ be the subgroup of 
$\T$ generated by $e^{2 \pi i < r, l>}, l \in \Z^{d}$, which 
is evidently finite. Suppose that $f(x_{0}) = 0$. It follows 
from our hypothesis that, for any $l$ in $\Z^{d}$, 
$e^{2 \pi i f} (\varphi^{l}(x_{0}))$ is in $F$. As $f$ is 
continuous and the orbit of $x_{0}$ is dense, we see that 
the range of $e^{2 \pi i f}$  is contained in $F$.
Let $k$ be the least positive integer such that $kr$ 
is in $\Z^{d}$. It follows that the range of $e^{2 \pi i k f}$ 
is simply $1$, so $kf$ is in $C(X, \Z)$.

It is clear that the function $\theta$ is continuous and 
\[
e^{2 \pi i \theta(x,l)} = e^{2 \pi i f} \circ \varphi
   e^{-2 \pi i f} e^{-2 \pi 1 <r,l> } = 1,
   \]
   so it is integer-valued. Finally, $\theta$ satisfies the 
   cocycle condition because it is the 
   sum of a coboundary 
   and a homomorphism (even if neither are integer-valued). 
   
   For the next statement, we observe that 
   $\sum_{i=1}^{d} kr_{i}\varepsilon_{i}(l) = 
   \sum_{i=1}^{d} kr_{i}l_{i} = < k<r,l>$ and that 
   $k\theta = d(kf) - \sum_{i=1}^{d} kr_{i}\varepsilon_{i}$.

For the last part, we compute 
\[
\tau_{\mu}^{1}[\theta]_{i} = \int_{X} \theta(x, e_{i}) 
\int_{X} \left( f(\varphi^{e_{i}}(x) - f(x) - r_{i}  \right) 
 d\mu(x)  = r_{i}.
 \]
Thus the  second statement holds.

Now let us assume that the second statement holds and 
prove the first does also. 
First, it is clear that $r$ is in $\Q^{d}$. 
Secondly, using the fact that $\theta$ takes integer values, 
exponentiating $2 \pi i$ times the equation in part 2 yields the 
equation of part 1.
\end{proof}

\begin{thm}
\label{rational:11}
Let $(X, \varphi)$ be a minimal, free $\Z^{d}$-Cantor
 system, $x_{0}$ be in $X$ and let $\mu$ be a $\varphi$-invariant measure. 
\begin{enumerate}
\item
An element $r \in \Q^{d}$ is in 
 $ \tau_{\mu}^{1}( \Q(H^{1}(X, \varphi)) )$ if and only
 if there is a continuous function 
 $\xi: X \rightarrow \T$ such that $\xi(x_{0})=1$ and
 \[
 \xi \circ \varphi^{l} = e^{2 \pi i <r, l>} \xi, l \in \Z^{d}.
 \]
 That is, $(e^{2 \pi i r_{1}}, \ldots, e^{2 \pi i r_{d}})$
  is a rational eigenvalue
 of $\varphi$ with continuous eigenfunction.
 \item 
 If $\xi : X \rightarrow \T$ is a continuous function with 
 $\xi(x_{0})=1$  and 
 $r \in \Q^{d}$ satisfy 
 \[
 \xi \circ \varphi^{l} = e^{2 \pi i <r, l>} \xi, l \in \Z^{d}.
 \]
 then $\xi(X)$ is a finite subgroup of $\T$.
 \item
 Then
  $ \tau_{\mu}^{1}( \Q(H^{1}(X, \varphi)) $ is a subgroup
  of $\Q^{d}$ and  is independent of $\mu$.
  \item
$ \tau_{\mu}^{1}: \Q(H^{1}(X, \varphi)) 
 \rightarrow \Q^{d}$
 is an isomorphism to its image.
 \end{enumerate}
\end{thm} 

\begin{proof}
For the first part, begin with $r$ in 
$ \tau_{\mu}^{1}( \Q(H^{1}(X, \varphi)) )$. This means that 
we can find $\theta$ in $Z^{1}(X, \varphi)$, a positive integer 
$k$, $j$ in $\Z^{d}$ and $g$ in $C(X, \Z)$ such that 
\[
k \theta  = \sum_{i=1}^{d} j_{i} \varepsilon_{i} + d(g),
\]
with $\tau_{\mu}^{1}[\theta] = r$. The last statement means that 
$\int_{X} \theta(x, e_{i}) d\mu(x) = r_{i}, 1 \leq i \leq d$.
Using $l= e_{i}$ and integrating in the equation above gives
$k r_{i} = j_{i}$, for all $i$. Then $\theta, r$ and 
$f = k^{-1}g$ satisfies the second condition of the last lemma.
We conclude that $\xi = e^{2 \pi i f}$ is a continuous 
eigenfunction with eigenvalue 
$(e^{2 \pi i r_{1}}, \ldots, e^{2 \pi i r_{d}})$ 
as claimed.

Conversely, if 
$(e^{2 \pi i r_{1}}, \ldots, e^{2 \pi i r_{d}})$ is a rational 
eigenvalue with continuous eigenfunction 
$\xi$, then as $X$ is totally disconnected, it is 
a simple matter to find a continuous real-valued function
$f$ with $e^{2 \pi i f} = \xi$. Then for any $x_{0}$ in 
$X$, the function $f - f(x_{0})$ and vector $r$
 satisfy the   hypotheses of part 1 of the last lemma. The desired 
 conclusion holds from part 2 of the lemma.
 
For the second part, it is clear that if $\xi(x) =1$, for some $x$, then
the orbit of $x$ is mapped into the finite subgroup  
$\{ e^{2 \pi i <r, l>} \mid  l \in \Z^{d} \}$. The conclusion 
follows from the 
continuity of $\xi$ and the minimality of $\varphi$.
 
 The third part is clear. For the last part, suppose that 
 $\theta$ and $\eta$ are in $Z^{1}(X, \varphi)$ and 
 $\tau_{\mu}^{1}[\theta] = \tau_{\mu}^{1}[\eta]$. Find positive 
 integers $k,l$ and vectors $j,m$ in $\Z^{d}$ such that 
  $k [ \theta] = \sum_{i} j_{i} [\varepsilon_{i}], 
  l [ \eta] = \sum_{i} m_{i} [\varepsilon_{i}]$. 
  Evaluating $\tau_{\mu}^{1}$ on each yields
  \[
 r_{i} =  \tau_{\mu}^{1}[\theta] = k^{-1} j_{i} = 
  \tau_{\mu}^{1}[\eta] = l^{-1} m_{i} 
  \]
  It follows that 
  \[
  kl [\theta] = l \sum_{i} j_{i} [\varepsilon_{i}], 
   k \sum_{i} m_{i} [\varepsilon_{i}] =k l [ \eta].
   \]
   It follows from Theorem \ref{rational:3} that 
   $[\theta] = [\eta]$.
\end{proof}

For the moment, we will let $G$ be a countable abelian group. 
In fact, all of our groups will be abelian, so we will use additive notation.
An action $(X, \varphi)$ of $G$ is a \emph{group rotation} if 
$X$ is a (compact) abelian group and there is a
 group homomorphism, also denoted $\varphi: G \rightarrow X$ such 
 that $\varphi^{g}(x) = x + \varphi(g)$, for all $x$ in $X$, $g$ in $G$.
 The system is minimal if and only if $\varphi(G)$ is dense in $X$.

Recall that if $\pi: (X, \varphi) \rightarrow (Y, \psi)$ is a factor
map between actions of an abelian group $G$, we may define 
\[
E_{\pi} = \{ (x_{1}, x_{2}) \in X \times X \mid \pi(x_{1}) = \pi(x_{2}) \}.
\]
It is clear that $E_{\pi}$ is an equivalence relation, is a closed 
subset of $X \times X$ and is invariant under $\varphi^{g} \times \varphi^{g}$, 
for every $g$ in $G$.
Conversely, if $E$  satisfies these three 
conditions, then the quotient space $X/E$ is Hausdorff and 
obtains an action of $G$ 
in an obvious way and the quotient map, $\pi_{E}: X \rightarrow X/E$ 
becomes a factor map.

Recall also (\cite{Aus}), that an action, $(Y, \psi)$, of $G$ is 
\emph{equicontinuous} if, for any open set, $U$ in $Y \times Y$
containing the diagonal, $\Delta_{Y} = \{ (y, y) \mid y \in Y \}$, 
there is another open set $V$, also containing $\Delta_{Y}$ such that
$(\psi \times \psi)^{g}(V) \subseteq U$, for every $g$ in $G$.
It is easy to see that every group rotation is 
equicontinuous and that any factor of an 
equicontinuous system is also equicontinuous.
If  $\pi: (X, \varphi) \rightarrow (Y, \psi)$ is a factor
map and $(Y, \psi)$ is equicontinuous, we also say that $\pi$ and
$E_{\pi}$ are 
equicontinuous.

Given  the system $(X, \varphi)$, one may consider the family 
of all  closed, $\varphi^{g} \times \varphi^{g}$-invariant, 
for all $g$ in $G$, equivalence relations
which are equicontinuous. Noting that $X \times X$ is in this family
and that the given properties are preserved under intersections, one 
obtains a minimal  closed, $\varphi \times \varphi$-invariant, 
equicontinuous equivalence relation, which we denote by $E_{eq}$.
The minimality of $E_{eq}$ means that its associated factor, which we denote
$\pi_{eq}:(X, \varphi) \rightarrow (X_{eq}, \varphi_{eq})$, 
is the \emph{maximal equicontinuous factor} of $(X, \varphi)$.

Next, we note that every continuous eigenfunction of 
$(X, \varphi)$,  
$\xi: X \rightarrow \T$  with eigenvalue
$\gamma$ in $\hat{G}$,  factors through $(X_{eq}, \varphi_{eq})$ as follows.
We can 
simply regard $\gamma$ as a group homomorphism from $G$ into $\T$ and
so $\xi$ is a factor map from $(X, \varphi)$ to the group rotation
$(\T, \gamma)$. As the latter is equicontinuous and $(X_{eq}, \varphi_{eq})$
is maximal,  we have a factor map
$\chi: (X_{eq}, \varphi_{eq}) \rightarrow (\T, \gamma)$ such that 
$\chi \circ \pi_{eq} = \xi$.

Now, we use the fact that every minimal
equicontinuous action of a countable, abelian group 
$G$, $(X, \varphi)$, is a group rotation. 
We refer the reader to \cite{Ku} or \cite{Wa} 
 for a proof of this fact (although 
the latter deals only with the case $G = \Z$). The  idea
is to pick any $x_{0}$ in $Y$ and define $\alpha(g) = \varphi^{g}(x_{0})$, so 
that the orbit of $x_{0}$ obtains the structure of a group. Then one shows
 that the equicontinuity of the action
implies this group structure can be extended to the closure, which is $X$.

We now observe that, for a minimal group rotation $(Y, \psi)$, 
a continuous function $\chi: Y \rightarrow \T$ is a character of
$Y$ if and only if $\chi(0)=1$ and it is a continuous eigenfunction
for $\psi$. The 'only if' direction is trivial. For the 'if' direction, we 
observe that we have 
 \begin{eqnarray*}
 \chi( y + \psi(g)) & = & <\!< g, \gamma > \! > \chi(y)  \\
  & = &  <\!< g, \gamma > \! > \chi(0) \chi(y)  \\
   & = & \chi(0 + \psi(g)) \chi(y) \\
   & = &  \chi(y)\chi(\psi(g))
   \end{eqnarray*}
   holds for every $g$ in $G$ and $y$ in $Y$. The fact 
   that $\psi(G)$ is dense in $Y$ and $\chi$ is continuous
    means that it holds if we replace $\psi(g)$ with
     any $y'$ in $Y$; i.e. $\chi$ is a character.

Next, we consider the connected subgroup of the identity in $X_{eq}$, 
which we
denote by $X_{eq}^{0}$ and the quotient $X_{eq}/X_{eq}^{0}$, 
which is compact and totally disconnected.  The easiest way to understand 
this is to examine that dual map to the quotient map: 
$q: X_{eq} \rightarrow   X_{eq}/X_{eq}^{0}$. Any character
on $X_{eq}/X_{eq}^{0}$ is induced by  a character
 on $X_{eq}$ and it is a simple
matter to check that a character, $\chi$, of $X_{eq}$ passes to the 
quotient if and only if it satisfies the following equivalent conditions:
\begin{enumerate}
\item $\chi(X_{eq}^{0}) = \{ 1 \}$,
\item $\chi(X_{eq})$ is a finite subgroup of $\T$, 
\item  $\chi(X_{eq})$ is a subgroup of the roots of unity.
\end{enumerate}

We now restrict our attention to 
$G = \Z^{d}$ and consider $(X, \varphi)$, a minimal $\Z^{d}$-action 
action on the Cantor set. Let 
$\mu$ be an invariant measure for $\varphi$ and 
let $H = \tau_{\mu}^{1}( \Q(H^{1}(X, \varphi))$. We claim that 
$H/ \Z^{d}$ is isomorphic to the dual of $X_{eq}/X_{eq}^{0}$. 
Let $r$ be any element of $H$. we know from part 1 of Theorem  
\ref{rational:11} that there is a continuous function
$\xi: X \rightarrow \T$ with $\xi(x_{0}) =1$ and
 \[
 \xi \circ \varphi^{l} = e^{2 \pi i <r, l>} \xi,
 \]
 for all $l$ in $\Z^{d}$. It follows from part 2 of the 
 same result that $\xi$ has finite range. 
 Then from the discussion
  above, we
 know that $\xi = \chi \circ q \circ \pi_{eq}$, for some 
 character $\chi$ on $X_{eq}/X_{eq}^{0}$. The map sending $r$ to $\chi$ 
 is obviously trivial on $\Z^{d}$. It is a simple matter to check
 that this is an isomorphism as claimed and that 
 the dual of this map is a conjugacy between the $\Z^{d}$-actions. 
 We have proved the
 following.

\begin{thm}
\label{rational:15}
Let $(X, \varphi)$ be a minimal, free $\Z^{d}$-Cantor
 system. Choose $\mu$ to be a $\varphi$-invariant measure 
 and let $H =  \tau_{\mu}^{1}( \Q(H^{1}(X, \varphi))$.
 There is a factor map
 \[
 \pi: (X, \varphi) \rightarrow (Y_{H}, \psi_{H}).
 \]
 Moreover, $(Y_{H}, \psi_{H})$ is the maximal 
 totally disconnected,
 equicontinuous factor of $(X, \varphi)$.
\end{thm}
 
The factor map is called almost one-to-one if there is a point $y$ in 
$Y_{H}$ such that $\pi^{-1}\{ y \}$ is a single point. 
The system $(X, \varphi)$ is a so-called Toeplitz system if and only if it
is expansive and the map $\pi$ is almost one-to-one - see \cite{Co}.

\section*{Acknowledgements}
The authors would like to thank R. Grigorchuk and J. Hunton
 for helpful 
conversations and  the Centre de Recerca Matem\'{a}tica 
in Barcelona for their hospitality and stimulating environment
during the program 'Operator Algebras: Dynamics and Interactions',
where much of this work was done.

\bibliographystyle{model1-num-names}

\end{document}